\documentclass[a4paper,12pt]{article}
\usepackage{amsmath}
\usepackage{amssymb}
\usepackage{amsopn}
\usepackage{amsthm}
\usepackage{amstext}
\usepackage{amsfonts}
\usepackage{graphicx}
\usepackage{color}
\newtheorem{lemma}{Lemma}[section]
\newtheorem{theorem}[lemma]{Theorem}
\newtheorem{definition}[lemma]{Definition}
\newtheorem{remark}[lemma]{Remark}
\newtheorem{corollary}[lemma]{Corollary}
\def\C{\mathbb{C}}
\def\R{\mathbb{R}}
\def\Q{\mathbb{Q}}
\def\Re{\mathrm{Re}\,}
\def\tr{\mathrm{tr}\,}
\def\Rtr{\mathrm{Retr}\,}
\def\rank{\mathrm{rank}\,}
\def\diag{\mathrm{diag}\,}
\def\vec{\mathrm{vec}}
\def\Im{\mathrm{Im}\,}
\def\Jm{\mathrm{Jm}\,}
\def\Km{\mathrm{Km}\,}
\def\dprod{\prod\limits}
\def\dwedge{\mathop{\bigwedge}\limits}
\def\etr{\mathrm{etr\,}}
\def\re{\mathrm{re}\,}

\def\vol{\mathrm{vol}\,}

\begin{document}
\title{{Zonal} polynomials and hypergeometric functions of quaternion matrix argument
\thanks{Project supported by Natural Science Foundation of China (no.10771069) and Shanghai Leading Academic
Discipline Project(no.B407)}}
\author{Fei Li\\
Department of Mathematics\\
East China University of Science and Technology, Shanghai 200237, P.R. China\\
\ \\
Yifeng Xue\thanks{Corresponding author\newline{\hspace*{0.5cm}}{\bf Keywords.}\ Zonal polynomial, Hypergeometric
function, quaternion matrix, Wishart matrix\newline{\hspace*{0.5cm}}{\bf 2000 Mathematics Subject Classification.}\
62H10, 60E10}\\
Department of Mathematics\\
East China Normal University, Shanghai 200240, P.R.China}
\date{}

\maketitle
\begin{abstract}
{{We define zonal}} polynomials of quaternion matrix argument and
deduce some important formulae of zonal polynomials and
hypergeometric functions of quaternion matrix argument. As an
application, we give the distributions of the largest {{and}}
smallest {{eigenvalues}} of {{a}} quaternion central {{Wishart}}
matrix $W\sim\mathbb{Q}W(n,\Sigma)${{,}} respectively.
\end{abstract}

\section{INTRODUCTION}

Zonal polynomials and hypergeometric {{functions}} of real (or
complex) symmetric matrices early introduced in \cite{her} and
\cite{james1, james2, james3} were used to study the density
functions and the distributions of eigenvalues of Wishart matrices.
Now they {{are}} very useful tools in the study of Multivariate
Statistical Analysis. There are many ways {{of defining }} zonal
polynomials. Some of them {{have}} appeared in \cite{her},
\cite{james1} and \cite{akimi}. Muirhead's definition of zonal
polynomials {{of a}} real matrix argument is an axiomatic definition
{{which}} appeared in \cite{robb}, involving partial differential
operators. This definition is {{easier and more convenient for}}
practical use. Gross and Richards {{defined}} zonal polynomials of
{{a}} matrix argument over the division algebra $\mathbf{F}$,
including the real {{and}} complex fields{{,}} and quaternion
division by means of the representation of groups. Maybe the authors
{{thought}} there were some problems in {{their}} results, {{since}}
they {{do not}} compute the numerical presentations of $C_\kappa(A)$
($A$ is a Hermitian quaternion matrix).

In this paper, we modify the definition of zonal polynomials of
{{a}} real matrix argument given in \cite{robb} and {{define}} zonal
polynomials of a quaternion matrix argument. Then we compute the
presentations of $C_\kappa(A)$. We also define quaternion
hypergeometric functions in terms of zonal polynomials {{of a}}
quaternion matrix argument and {{derive}} some useful formulas for
quaternion hypergeometric functions. Using these results, we give
the distributions of the largest {{and smallest eigenvalues of a
quaternion Wishart}}  matrix $W=A^HA$ (i.e., $W \sim \Q
W_m(n,\Sigma), n\geqslant m$, $A \sim\Q N_{n\times
m}(0,I_n\bigotimes \Sigma)$).

The paper is organized as follows. \S 2 provides the preliminary
tools for deriving our results. The zonal polynomials and
hypergeometric functions of {{a}} quaternion matrix argument will be
studied in \S 3 and \S 4{{,}} respectively. In {{the}} last section,
we will give the distributions of the largest {{and }} smallest
eigenvalues{{.}}

\section{PRELIMINARY }

{{Following}} \cite{zh}, let $\C$ and $\R$ denote the fields of
complex and real numbers{{,}} respectively{{,}} and let $\Q$
{{denote the}} quaternion division {{algebra}} over $\R$, i.e.,
every $a\in\Q$ can be expressed as $a=a_1+a_2i+a_3j+a_4k$, where
$i,j,k$ satisfy {{the}} following relations
$$
i^2=j^2=k^2=-1,\ ij=-ji=k,jk=-kj=i,ki=-ik=j.
$$
Put $a^H=a_1-a_2i-a_3j-a_4k$ and
$\|a\|=(a^Ha)^{1/2}=(a^2_1+a^2_2+a^2_3+a^2_4)^{1/2}$. Let
$\R^{m\times n}$, $\C^{m\times n}$, $\Q^{m\times n}$ denote the set
of all $m\times n$ matrices over $\R$, $\C$ and $\Q${{,}}
respectively. Any $A\in \Q^{m\times n}$ can be written as
$A=(a_{ij})_{m \times n}=A_1+A_2i+A_3j+A_4k$, where $a_{ij} \in\Q$,
{{and}} $A_1,~A_2,~A_3,~A_4\in\R^{m\times n}$. $A_1$ is the real
part of $A$, denoted by $\Re A$. We also set $\Im(A)=A_2,\
\Jm(A)=A_3$ and $\Km(A)=A_4$. Put $A^H=(a_{ji}^H)_{n \times
m}=A_1^{'}-A_2^{'}i-A_3^{'}j-A_4^{'}k$. We {{say}} $A$ is Hermitian
if $A^H=A$. The eigenvalues of {{a}} Hermitian matrix are all real.
If the eigenvalues are all positive, then we {{say}} it is a
positive definite quaternion matrix.

Let $\tr(\cdot)$ be the trace on $\Q^{n\times n}$ and put $\Rtr(A)=\tr(\Re A)$ for $A\in\Q^{n\times n}$. We have
$$
\Rtr(A)=\dfrac{1}{\,2\,}\tr(A+A^H),\ \Rtr(AB)=\Rtr(BA),\quad\forall\, A,\,B\in\Q^{n\times n}.
$$
{{Moreover,}} if $A=A^H\in\Q^{n\times n}$, then
$\Rtr(A)=\tr(A)=\sum\limits^n_{s=1}\lambda_s$, where
$\lambda_1,{{\ldots}}, \lambda_n$ are {{the}} eigenvalues of $A$.
Set ${_qS(n)}=\{A\in\Q^{n\times n}\vert\,A^H=A\}$ and
$$
{_qO(n)}=\{A\in\Q^{n\times n}\vert\,A^HA=AA^H=I_n\},\
{_qV_{n,m}}=\{A\in\Q^{m \times n}\vert\, A^HA=I_n\}.
$$

Let $A\in\Q^{m\times n}$ be
$A=A_{1}+A_{2}i+A_{3}j+A_{4}k=(A_{1}+A_{2}i)+(A_{3}+A_{4}i)j=B_{1}+B_{2}j$.
$A$ has the
complex representation
$A^{\sigma}=\begin {pmatrix}B_{1}&-B_{2}\\
\overline{B_{2}}&\overline{B_{1}}\end{pmatrix}$
${(\overline{B_1}=A_1-A_2i,\ \overline{B_2}=A_3-A_4i)}$ and the real
representations
$$
{_1A}=\left({\begin{array}{rrrr}
A_{1}&A_{2}&A_{3}&A_{4}\\
 -A_{2}&A_{1}&-A_{4}&A_{3}\\-A_{3}&A_{4}&A_{1}&-A_{2}\\-A_{4}&-A_{3}&A_{2}&A_{1}\end{array}}\right)
\quad\text{and}\ {_{2}A}=\left({\begin{array}{rrrr}A_{1}
&-A_{2}&-A_{3}&-A_{4}\\
A_{2}&A_{1}&-A_{4}&A_{3}\\A_{3}&A_{4}&A_{1}&-A_{2}\\A_{4}&-A_{3}&A_{2}&A_{1}\end{array}}
\right).
$$

For $A\in\Q^{m\times n}$, denote by $|A|_q=\det(A^\sigma)$ and
$|A|_d=|A^HA|$ the q--determinant and double determinant of $A${{,}}
respectively{{;}} here $|\cdot\,|$ is the determinant of a square
quaternion matrix given in \cite{zh}. We have $|A^H|_d=|A|_d$ and
$|A|_d=|A|_q$ (cf. \cite{zh}). Moreover, we have
\begin{lemma}\label{Lax}
Let $A\in\Q^{n\times n}$.
\begin{enumerate}
\item[$(1)$] $|A|_d^2=\det({_1A})=\det({_2A});$
\item[$(2)$] Let $A=T^HT$, where $T=(t_{ij})_{n\times n}\in \Q^{n \times n}$ is {{an}} upper-triangular
matrix with $t_{ii}>0${{,}} $i=1,{{\ldots}},m$. Then
$|A|=t^2_{11}\cdots t^2_{nn}$.
\end{enumerate}
\end{lemma}
\begin{proof}
(1) Set $S_1=\begin{pmatrix}A_{1}&-A_{3}\\ A_{3}&A_{1}\end{pmatrix}$, $S_2=\begin{pmatrix}A_{2}&-A_{4}\\-A_{4}&-A_{2}
\end{pmatrix}$. Then by the proof of \cite[Lemma 3.1]{math},
\begin{align*}
|A|_d^2=|A|_q^2&
=\Bigg\vert\det\begin{pmatrix}A_{1}+A_{2}i & -A_{3}-A_{4}i \\A_{3}-A_{4}i & A_{1}-A_{2}i \end{pmatrix}\Bigg\vert^2
=\det(S_1+S_2i)\det(S_1-S_2i)\\
&=\det\begin{pmatrix}S_1+S_2i\\ &S_1-S_2i\end{pmatrix}=\det\begin{pmatrix}2S_1&-iS_2\\ -iS_2&\frac{1}{2}S_1\end{pmatrix}.
\end{align*}
Note that
$$
\begin{pmatrix}2S_1&-iS_2\\ -iS_2&\frac{1}{2}S_1\end{pmatrix}=
\begin{pmatrix}2\\ &-i\end{pmatrix}\begin{pmatrix}S_1&S_2\\ -S_2&S_1\end{pmatrix}\begin{pmatrix}1\\ &\frac{i}{2}
\end{pmatrix}=\begin{pmatrix}2\\ &i\end{pmatrix}\begin{pmatrix}S_1&-S_2\\ S_2&S_1\end{pmatrix}
\begin{pmatrix}1\\ &-\frac{i}{2}\end{pmatrix}.
$$
So $|A|_d^2=\det\begin{pmatrix}S_{1}&S_{2}\\-S_{2}&S_{1}\end{pmatrix}=\det\begin{pmatrix}S_{1}&-S_{2}\\ S_{2}&S_{1}
\end{pmatrix}=\det({_1A})=\det({_2A})$.

(2) We have
$|A|=|T|_d=|T|_q=\det(T^\sigma)=\det\begin{pmatrix}T_1&-T_2\\
\overline{T_2}&\overline{T_1}\end{pmatrix}$, where $T=T_1+T_2j$ with
$T_1, T_2\in\mathbb{C}^{n\times n}$ and $T_1$, $T_2$ have the form
$$
T_1=\begin{pmatrix}t_{11}&\\ _{\mbox{\bf\huge 0}}&\ddots
&^{\mbox{\bf\huge *}}\\ & &t_{nn}\end{pmatrix},\quad
T_2=\begin{pmatrix}0&\\ _{\mbox{\bf\huge 0}}&\ddots
&^{\mbox{\bf\huge *}}\\ & & 0\end{pmatrix}{{,}}
$$
{{respectively. A simple }}computation shows that $\det\begin{pmatrix}T_1&-T_2\\
\overline{T_2}&\overline{T_1}\end{pmatrix} =t^2_{11}\cdots
t^2_{nn}$.
\end{proof}

Let $X=X_1+X_2i+X_3j+X_4k \in \mathbb{Q}^{m\times n}$ and
$X_1,\,X_2,\,X_3,\,X_4$ are $m\times n$ matrices of functionally
independent real variables. Define the volume of $X$ as
$(dX)=(dX_1)\bigwedge(dX_2)\bigwedge(dX_3)\bigwedge(dX_4)$, where
$(dX_s)${{,}} $s=1,2,3,4${{,}} are defined in \cite{math}.

\begin{lemma}\label{Lbx}
$X,~Y\in \mathbb{Q}^{m\times n}$ and $Y=AXB$ where
$A\in\mathbb{Q}^{m\times m}$ {{and}} $B\in\mathbb{Q}^{n\times n}$
are constant invertible matrices.
\begin{enumerate}
\item[$(1)$] We have $(dY)=|A|_q^{2n}|B|_q^{2m}(dX);$
\item[$(2)$] Suppose $X\in{_qS(m)}$ and $B=A^H$. Then $(dY)=|A|_q^{2m-1}(dX)$.
\end{enumerate}
\end{lemma}
\begin{proof}
(1) Let $Y=AW$ and $W=XB$. Then $dY=AdW$, $dW=dX\,B$ and
\begin{align*}
(dY'_{1},dY'_{2},dY'_3,dY'_{4})'&=(_2A)(dW'_{1},dW'_{2},dW'_3,dW'_{4})'\\
(dW_1,dW_2,dW_3,dW_4)&=(dX_1,dX_2,dX_3,dX_4)(_1B).
\end{align*}
Using the operator $\mathrm{vec}(\cdot)$ (defined in
\cite[Definition 1.2]{math}) to $dX_s,\ dY_s$ and $dW_s$,
$s=1,{{\ldots}},4$, we have
$$
\begin{pmatrix} \vec(dY_{1})\\ \vec(dY_{2})\\ \vec(dY_3)\\ \vec(dY_{4})\end {pmatrix}=
{_2(I\otimes A)}\begin {pmatrix}\vec(dW_{1})\\ \vec(dW_{2})\\ \vec(dW_3)\\ \vec(dW_{4})\end {pmatrix},\
\begin{pmatrix} \vec(dW_{1})\\ \vec(dW_{2})\\ \vec(dW_3)\\ \vec(dW_{4})\end {pmatrix}=
((_1B)'\otimes I)\begin {pmatrix}\vec(dX_{1})\\ \vec(dX_{2})\\ \vec(dX_3)\\ \vec(dX_{4})\end {pmatrix}
$$
by \cite[Lemma 1.1]{math} so that
$$
\begin{pmatrix}\vec(dY_{1})\\ \vec(dY_{2})\\ \vec(dY_3)\\ \vec(dY_{4})\end {pmatrix}=
{_2(I\otimes A)}((_1B)'\otimes I)\begin {pmatrix}\vec(dX_{1})\\ \vec(dX_{2})\\ \vec(dX_3)\\ \vec(dX_{4})\end {pmatrix}.
$$
Thus by \cite[Lemma 1.2]{math} and Lemma \ref{Lax},
$$
(dY)=|{_1A}|^n|{_2B}|^m(dX)=|A|_q^{2n}|B|_q^{2m}(dX).
$$

(2) Since $A$ is invertible, it follows from \cite[Theorem 4.3]{zh} that $A$ is the product of elementary
quaternion matrices. Thus using the same method as in the proof of \cite[Theorem 1.20]{math}, we can get the assertion.
\end{proof}

{{The following}} two lemmas, which come from {{\cite[p37,
p38]{Di}}}, will be used in {{this}} paper:
\begin{lemma}\label{lemma1.3}
$X\in{_qS(m)}$ with $X>0$. Suppose $X=T^HT$, where
$T=(t_{ij})_{m\times m}\in \Q^{m \times m}$ is {{an}}
upper-triangular matrix with real diagonal elements. Then
$$
(dX)=2^m\prod_{i=1}^{m}t_{ii}^{4(m-i)+1}(dT),
$$
where
$(dT)=\dwedge^m_{s=1}d\,t_{ss}\dwedge^4_{p=1}\dwedge^m_{s<t}d\,~t^{(p)}_{st}$,
$t_{st}=t^{(1)}_{st}+ t^{(2)}_{st}i+t^{(3)}_{st}j+t^{(4)}_{st}k$,
$s<t,~ t=1,{{\ldots}},m$.
\end{lemma}

\begin{lemma}\label{lemma1.6}
Let $Z=H_1T\in \mathbb{Q}^{n \times m}$ with $H_1\in{_qV_{m,n}}$,
here $T$ the upper triangular matrix with positive diagonal
elements.  Then we have
$$(dZ)=\prod_{i=1}^{m}t_{ii}^{4(n-i)+3}(dT)\wedge (H_1^HdH_1),$$
where $(H_1^HdH_1)=\dwedge^m_{s=1}\dwedge^n_{t=s+1}h_t^Hd\,h_s$ for
$H=(H_1|H_2)=(h_1,{{\ldots}},h_m|h_{m+1},{{\ldots}}, h_n)$.
\end{lemma}

In this paper, we shall use the {{singularvalue}} decomposition
(SVD) of a matrix in $\Q^{m\times n}$ as follows. Let
$A\in\Q^{m\times n}$ with $\rank A=r$. Then there are $U=(U_1\vert
U_2)\in{_qO(m)}$, $V=(V_1\vert V_2)\in{_qO(n)}$, with
$U_1\in{_qV_{r,m}},\ V_1\in{_qV_{r,n}}$ such that
\begin{equation}\label{1ex}
A=U\begin{pmatrix} D&0\\ 0&0\end{pmatrix}V^H=U_1DV_1^H
\end{equation}
(\cite[Theorem 7.2]{zh}), where
$D=\diag(\lambda_1,{{\ldots}},\lambda_r)$ and
$\lambda_1,{{\ldots}},\lambda_r$ are the singular values of $A$. If
$A\in{_qS}(n)$ with $\rank A=r$, then $V$ and $V_1$ can be taken as
$U$ and $U_1$ in (\ref{1ex}) respectively.
\begin{lemma}\label{lemma1.5}
Let $X\in\Q^{m\times n}$ with $\rank X=n\le m$. Let $X=UDV^H$ with
$U\in{_qV_{n,m}}$, $V\in{_qO(n)}$ and
$D=\diag(\lambda_1,{{\ldots}},\lambda_n)$ (assume that
$\lambda_1>\lambda_2>\cdots>\lambda_n>0$). Then
\begin{enumerate}
\item[$(1)$] $(dX)=(2\pi^2)^{-n}\dprod_{j<i}^{n}(\lambda^2_{j}-\lambda^2_{i})^{4}\dprod^n_{i=1}\lambda_i^{4m-4n+3}
(dD)\dwedge(U^{H}dU)\dwedge(V^HdV)$ for $U\not=V;$
\item[$(2)$] $(dX)=(2\pi^2)^{-n}\dprod_{j<i}^{n}(\lambda_{j}-\lambda_{i})^{4}(dD)\dwedge(U^{H}dU)$ for $m=n$ and
$U=V$,
\end{enumerate}
where $(V^HdV)=\dwedge^n_{s<t}v^H_td\,v_s$ for $V=(v_1\cdots v_n)$,
$(U^HdU)=\dwedge^n_{s<t}u^H_td\,u_s$ for $U=(u_1\cdots u_n)$.
\end{lemma}
\begin{proof}The assertions can be found in \cite[p241, p242]{ER}. But {{we}} must divide the volume elements by
$(2\pi^2)^n$ to normalize the arbitrary phases of elements in the
first row of $U$.
\end{proof}
\begin{corollary}\label{coro1.1}
Let $X=UDV^{H}$ with $X,~U, ~D,~V$ given in Lemma \ref{lemma1.5}. Put $Z=X^{H}X$. Then
$$
(dX)=2^{-n}\prod_{i=1}^{n}\lambda_{i}^{4m-4n+2}(dZ)\wedge(U^{H}dU)=2^{-n}|X|_q^{2m-2n+1}(dZ)\wedge(U^{H}dU).
$$
\end{corollary}
Recall that a quaternion variable $X=X_1+X_2i+X_3j+X_4k\sim
\mathbb{Q}N(0,1)$ if $X_1, X_2, X_3, X_4$ iid. $N(0,\frac{1}{4})$.
Thus $X=(x_{ij})_{n\times m}\in \Q^{n\times m}$ is {{said}} to be
the quaternion {{normal}} matrix $\Q N_{n\times m}(0,I_n\otimes
I_m)$ (or $X\sim\Q N_{n \times m}(0,I_n\otimes I_m)$) if
$\{x_{ij}\vert\,i=1,{\ldots},n,j=1,{\ldots},m\}$ iid. to $\Q
N(0,1)$. It is easy to deduce that the density function of $X \sim\Q
N_{n \times m}(0,I_n\otimes I_m)$ is
\begin{equation}\label{eqa}
f(X)=\dfrac{2^{2mn}}{\pi^{2mn}}\exp(-2\tr(X^HX)).
\end{equation}
 By (\ref{eqa}) and {{Lemma}} \ref{lemma1.6}, we can get
 \begin{equation}{{\vol(V_{m,n})}}=\int_{V_{m,n}}(H_1^HdH_1)=\frac{2^m\pi^{2mn-m^2+m}}{\dprod_{i=1}^m\Gamma[2n-2(i-1)]}
 =\frac{2^m\pi^{2mn}}{\mathbb{Q}\Gamma_m(2n)}\end{equation}
  where
 $\mathbb{Q}\Gamma(a)=\pi^{m^2-m}\dprod_{i=1}^m\Gamma(a-2(i-1))$ $({{\Re(a)}}>2(m-1))$ (cf. (4.1) of \cite{Di}).

 We call $Y\sim \mathbb{Q}N_{n \times m}(\mu,I_n\otimes \Sigma)$ if $Y= \mu+XB^H$, where
 $X \sim \mathbb{Q}N_{n \times m}(0,I_n\otimes I_m), \Sigma=BB^H$ is invertible.
 By {{Lemma}} \ref{Lax} and (\ref{eqa}), we can write the
density function of $Y\sim \mathbb{Q}N_{n \times m}(\mu,I_n\otimes \Sigma)$ as follows:
\begin{equation}
\frac{2^{2mn}}{\pi^{2mn}|\Sigma|^{2n}}\exp(\Rtr(-2\Sigma^{-1}(Y-M)^H(Y-M))).
\end{equation}
Let $W=Y^HY$, we say $W\sim \mathbb{Q}W_m(n,\Sigma)$ $(n\geqslant
m)$, if $Y \sim \mathbb{Q}N_{n \times m}(0,I_n \otimes \Sigma)$. $W$
is called the quaternion central {{Wishart}} matrix and the density
function of $W$ is
\begin{equation}\label{equaaa}
\frac{2^{2mn}}{\mathbb{Q}\Gamma_m(2n)|\Sigma|^{2n}}\exp(\Rtr(-2\Sigma^{-1}W))|W|^{2n-2m+1}.
\end{equation}

As applications of the theory of zonal polynomials of quaternion
matrix argument, we discuss the distributions of the maximum and the
minimum eigenvalues of $W${{,}} ~respectively{{,}} in the last
section.

\section{ZONAL POLYNOMIAL FOR QUATERNION MATRIX}

The zonal polynomials of a Hermitian matrix are defined in terms of
partitions of positive integers. Let $k$ be a positive integer; a
partition $\kappa$ of $k$ is written as $\kappa=(k_1,k_2,\cdots)$,
where $\sum_ik_i=k$, with the {convention{,}} unless otherwise
stated, that $k_1\geqslant k_2 \geqslant \cdots$, where
$k_1,k_2,\cdots$ are non-negative integers. And if
$\kappa=(k_1,k_2,\cdots)$ and $\lambda=(l_1,l_2,\cdots)$ are two
partitions of $k$,  we will write $\kappa> \lambda$ if $k_i>l_i$ for
the first index $i$ for which the parts are unequal.
\begin{definition}\label{def2.1}
Let $Y\in{_qS(m)}$ with eigenvalues $y_1,y_2,{\ldots},y_m$ and let
$\kappa=(k_1,k_2,\cdots)$ be a partition of $k$ into not more than
$m$ parts. The zonal polynomial of $Y$ corresponding to $\kappa$,
denoted by $C_{\kappa}(Y)$ (in this paper, we use the symbol
$C_{\kappa}(Y)$  to denote the zonal {{polynomials}} of Hermitian
quaternion matrices for notational simplicity) is a symmetric
homogeneous polynomial of degree $k$ in the latent roots
$y_1,{\ldots}, y_m$ such that:
\begin{enumerate}
\item[\rm{(i)}] The term of highest weight in $C_{\kappa}(Y)$ is
$y_1^{k_1},\cdots, y_m^{k_m}$, that
is,\begin{equation}C_{\kappa}(Y)=d_{\kappa}y_1^{k_1}\cdots
y_m^{k_m}+ \textit{terms of lower weight} \end{equation} where
$d_{\kappa}$ is a constant.
\item[\rm{(ii)}] $C_{\kappa}(Y)$ is an eigenfunction of the differential
operator $\Delta_Y$ given by
\begin{equation}
\Delta_Y=\sum_{i=1}^my_i^2\frac{\partial^2}{\partial
y_i^2}+\sum_{i=1}^{m}\sum_{j=1,j\neq
i}^m4\frac{y_i^2}{y_i-y_j}\frac{\partial}{\partial
y_i}\end{equation}
\item[\rm{(iii)}] As $\kappa$ varies over all partitions of $k$, the zonal
polynomials have unit coefficients in the expansion of $(\tr Y)^k$, that is
\begin{equation}
(\tr Y)^k=(y_1+y_2+\cdots+y_m)^k=\sum_{\kappa}^mC_{\kappa}(Y).
\end{equation}
\end{enumerate}
\end{definition}

By the way, if we replace $(ii)$ by $(ii)^{'}$:
\vspace{1mm}

$(ii)^{'}$ $C_{\kappa}(Y)$ is an eigenfunction of the differential
operator $\Delta_Y$ given by
\begin{equation}
\Delta_Y=\sum_{i=1}^my_i^2\frac{\partial^2}{\partial
y_i^2}+\sum_{i=1}^{m}\sum_{j=1,j\neq
i}^m2\frac{y_i^2}{y_i-y_j}\frac{\partial}{\partial
y_i}\end{equation}

Then the conditions $(i),(ii)^{'}$ and $(iii)$ {{define}} zonal
polynomials for Hermitian complex matrices. We can verify this
definition of zonal polynomials is just coincide with the definition
of zonal polynomials for Hermitian complex matrices in
\cite{james2}.

By using the same method as in the proof of \cite[Theorem
7.2.2]{robb}, we can obtain {{the}} following:
\begin{lemma}\label{th2.1}
The zonal polynomial $C_{\kappa}(Y)$ corresponding to the partition
$\kappa=(k_1,k_2,{\ldots},k_m)$ of $k$ satisfies the partial
differential equation
\begin{equation}
\Delta_YC_{\kappa}(Y)=[\rho_{\kappa}+k(4m-1)]C_\kappa(Y)
\end{equation}
where $\Delta_Y$ is given by (7) and
\begin{equation}
\rho_\kappa=\sum_{i=1}^mk_i(k_i-4i)
\end{equation}
\end{lemma}

If $\kappa=(k_1,k_2,{\ldots},k_m)$, the monomial symmetric function
of $y_1,y_2,{\ldots},y_m$ corresponding to $\kappa$ is defined as
$M_{\kappa}=y_1^{k_1}\cdots y_m^{k_m}+\textit{symmetric terms}$.
For example,
$$
M_1(Y)=y_1+\cdots+y_m,\quad M_2(Y)=y_1^2+\cdots+y_m^2,\quad M_{1,1}(Y)=\sum_{i<j}^m y_iy_j
$$
and so on.

When $k=1$, $C_{(1)}=\tr Y=y_1+\cdots+y_m$ by (8).

When $k=2$, there is two partitions $(1,1)$, $(2,0)$ by
definition \ref{def2.1} and Lemma \ref{th2.1}, so we have following equations,
\begin{align}
C_{(2)}=&d_{(2)}M_{(2)}(Y)+\beta M_{(1,1)}(Y)\\
C_{(1,1)}=&(2-\beta)M_{(1,1)}(Y)\\
\Delta_YC_{(2)}(Y)=&(8m-6)C_{(2)}(Y)
\end{align}
We have $d_{(2)}=1$ from above, since $C_{(2)}+C_{(1,1)}=(\tr Y)^2$.  Also we can verify
\begin{align}
\Delta_YM_{(2)}(Y)=&(8m-6)M_{(2)}(Y)+8M_{(1,1)}(Y)\\
\Delta_YM_{(1,1)}(Y)=&(8m-12)M_{(1,1)}(Y).
\end{align}
By means of (15), (16) and (14), we have $\beta=\dfrac{4}{\,3\,}$ by
the following equation,
$$
(8m-6)(M_{(2)}(Y)+\beta M_{(1,1)}(Y))=(8m-6)M_{(2)}(Y)+8M_{(1,1)}(Y)+(8m-12)\beta M_{(1,1)}(Y).
$$
Then {{the}} two zonal polynomials for Hermitian quaternion matrices
in the case $k=2$ are
$$
C_{(2)}=M_{(2)}(Y)+\frac{4}{3}M_{(1,1)}(Y),\quad C_{(1,1)}=\frac{2}{3}M_{(1,1)}(Y).
$$

Now {{we}} consider the case $k=3$. We {{have}} three partitions
$(3),(2,1),(1,1,1)$ when $k=3$. Thus,
\begin{align*}
C_{(3)}=&M_{(3)}(Y)+\beta M_{(2,1)}(Y)+\gamma M_{(1,1,1)}(Y)\\
C_{(2,1)}=&(3-\beta)M_{(2,1)}(Y)+\delta M_{(1,1,1)}(Y)\\
C_{(1,1,1)}=&(6-\gamma-\delta)M_{(1,1,1)}(Y).
\end{align*}
Since
\begin{align*}
\Delta_YM_{(3)}(Y)=&(12m-6)M_{(3)}(Y)+12M_{(2,1)}(Y)\\
\Delta_YM_{(2,1)}(Y)=&(12m-14)M_{(2,1)}(Y)+24M_{(1,1,1)}(Y)\\
\Delta_YM_{(1,1,1)}(Y)=&12(m-2)M_{(1,1,1)}(Y),
\end{align*}
it follows from Lemma \ref{th2.1} that
$$
\Delta_YC_{(3)}(Y)=(12m-6)C_{(3)}(Y),\quad \Delta_YC_{(2,1)}(Y)=(12m-14)C_{(2,1)}(Y).
$$
From {{the}} above equations, we can deduce that
$\beta=\dfrac{3}{\,2\,},\ \gamma=2,\ \delta=\dfrac{18}{\,5\,}$.
Therefore, we have three zonal polynomials for Hermitian quaternion
matrices when $k=3$ as follows:
\begin{align*}
C_{(3)}(Y)=&M_{(3)}(Y)+\frac{3}{\,2\,} M_{(2,1)}(Y)+2M_{(1,1,1)}(Y)\\
C_{(2,1)}(Y)=&\frac{3}{\,2\,}M_{(2,1)}(Y)+\frac{18}{\,5\,}M_{(1,1,1)}(Y)\\
C_{(1,1,1)}(Y)=&\frac{2}{\,5\,}M_{(1,1,1)}(Y).
\end{align*}

In general, let $\kappa$ be a partition of $k$. Then $C_\kappa(Y)$
can be expressed in terms of monomial symmetric functions as
$$
C_\kappa(Y)=\sum \limits_{\lambda\leqslant \kappa}c_{(\kappa,\lambda)}M_{(\lambda)}(Y).
$$
By Lemma \ref{th2.1}, we obtain {{that}} the coefficients
$c_{(\kappa,\lambda)}$ {{are}} determined by {{the}} following
equation:
\begin{equation}\label{eqxxx}
c_{(\kappa,\lambda)}=\sum_{\lambda < \mu \leqslant
\kappa}\frac{4[(l_i+t)-(l_j-t)]}{\rho_\kappa-\rho_\lambda}c_{(\kappa,\mu)},
\end{equation}
where $\rho_\kappa=\sum\limits_{i=1}^mk_i(k_i-4i)$,
$\lambda=(l_1,{\ldots},l_m)$ and
$\mu=(l_1,{\ldots},l_i+t,{\ldots},l_j-t,{\ldots},l_m)$ for
$t=1,\cdots,l_j$ such that, when the parts of the partition $\mu$
are arranged in descending order, $\mu$ is above $\lambda$ and below
or equal to $\kappa$. The summation in (\ref{eqxxx}) is over all
such $\mu$, including possibly, non--descending ones, and any empty
sum is taken to be zero.

For example, when $k=4$, we have five partitions
$(4),(3,1),(2,2),(2,1,1),(1,1,1,1)$. Then the zonal polynomial $C_{(4)}(Y)$ has the form
\begin{align*}
C_{(4)}(Y)=&M_{(4)}(Y)+c_{(4)(3,1)}M_{(3,1)}(Y)+c_{(4),(2,2)}M_{(2,2)}(Y)\\
&\ +c_{(4),(2,1,1)}M_{(2,1,1)}(Y)+c_{(4),(1,1,1,1)}M_{(1,1,1,1)}(Y).
\end{align*}
By (11), we have
$$
\rho_{(4)}=0,\ \rho_{(3,1)}=-10,\ \rho_{(2,2)}=-16,\
\rho_{(2,1,1)}=-22,\ \rho_{(1,1,1,1)}=-36.
$$
Let $\kappa=(4)$, $\lambda=(3,1)$. Then by (\ref{eqxxx}),
$c_{(4)(3,1)}=\dfrac{4 \times 4}{10}\times 1=\dfrac{8}{\,5\,}$.
The coefficient $c_{(4),(2,2)}$ comes from the partitions
$(3,1),(4)$, so
$$
c_{(4)(2,2)}=\frac{4 \times 2}{16}\times \frac{8}{\,5\,}+\frac{4 \times 4}{16}\times 1=\frac{9}{\,5\,}.
$$
Since the coefficient $c_{(4),(2,1,1)}$ comes from the partitions
$(3,1,0),(3,0,1),(2,2,0)$,
$$
c_{(4),(2,1,1)}=2\times \frac{4\times 3}{22}\times \frac{8}{\,5\,}+\frac{4\times 2}{22}\times \frac{9}{\,5\,}
=\frac{12}{\,5\,}.
$$
Noting that the coefficient $c_{(4),(1,1,1,1)}$ comes from the partitions
$(2,0,1,1)$, $(2,1,0,1)$, $(2,1,1,0)$, $(1,2,1,0)$, $(1,2,0,1)$, $(1,1,2,0)$, we have
$$
c_{(4),(1,1,1,1)}=6 \times \frac{4 \times 2}{36}\times \frac{12}{5}=\frac{16}{5}.
$$

We list the coefficients of $M_{\lambda}(Y)$ in $C_\kappa(Y)$ for
quaternion matrix $Y$ in {{the}} Table. We {{see that}} these
coefficients are different from {{these in}} the real cases given in
\cite[{p238}]{robb}. \vspace*{2mm}

\begin{tabular}{l}
Table: Coefficients of monomial symmetric functions $M_\lambda(Y)$ in $C_\kappa(Y)$  \\
\hline\\

$k=2$,\\

\begin{tabular}{cc}
&$\lambda$\\$\kappa$& \begin{tabular}{c|cc}&(2)&(1,1)\\
\hline (2)&1&4/3 \\  (1,1)&0&2/3\end{tabular}
\end{tabular} \\
\ \\

$k=3$,\\

\begin{tabular}{cc}
&$\lambda$\\$\kappa$&
\begin{tabular}{c|ccc}&(3,0)&(2,1)&(1,1,1)\\
\hline (3,0)&1&3/2&2 \\
(2,1)&0&3/2&18/5\\(1,1,1)&0&0&2/5\end{tabular}
\end{tabular}\\
\ \\
$k=4$,\\

\begin{tabular}{cc}
&$\lambda$\\$\kappa$&
\begin{tabular}{c|ccccc}&(4)&(3,1)&(2,2)&(2,1,1)&(1,1,1,1)\\
\hline (4)&1&8/5&9/5&12/5&16/5 \\
(3,1)&0&12/5&16/5&104/15&64/5\\
(2,2)&0&0&1&4/3&16/5\\(2,1,1)&0&0&0&4/3&32/7\\
(1,1,1,1)&0&0&0&0&8/35\end{tabular}
\end{tabular}\\
\ \\
$k=5$,\\

\begin{tabular}{cc}
&$\lambda$\\$\kappa$&
\begin{tabular}{c|ccccccc}&(5)&(4,1)&(3,2)&(3,1,1)&(2,2,1)&(2,1,1,1)&(1,1,1,1,1)\\
\hline (5)&1&5/3&2&8/3&3&4 &16/3\\
(4,1)&0&10/3&5&220/21&90/7&160/7&800/21\\(3,2)&0&0&3&4&26/3&16&32\\
(3,1,1)&0&0&0&20/7&80/21&85/7&200/7\\(2,2,1)&0&0&0&0&5/3&4&80/7\\
(2,1,1,1)&0&0&0&0&0&1&40/9\\(1,1,1,1,1)&0&0&0&0&0&0&8/63\end{tabular}
\end{tabular} \\
\hline\\
\end{tabular}
%\newpage
\vspace*{2mm}

Let $X$ be an $m \times m$ positive definite quaternion matrix and
put
\begin{equation}\label{eqzzz}
(ds)^2=\Rtr(X^{-1}dXX^{-1}dX)
\end{equation}
where $dX=(dx_{ij})_{m\times m}$. This is a differential form and is invariant under the transformation
$X\rightarrow LXL^H$, here $L\in\Q^{m \times m}$ is invertible. For then $dX\rightarrow LdXL^H$, so that
\begin{align*}
\Rtr(X^{-1}dXX^{-1}dX)&\rightarrow\Rtr((LXL^H)^{-1}LdXL^H(LXL^H)^{-1}LdXL^H)\\
&=\Rtr(X^{-1}dXX^{-1}dX).
\end{align*}

Put $n=2m^2-m$, let $x$ be the $n \times 1$ vector
\begin{align*}
x=(&x_{11},\Re x_{12},\ldots,\Re x_{1m},x_{22},\ldots,\Re x_{2m},\ldots,
x_{mm},\Im x_{12},\ldots,\Im x_{m ,m-1},\\
&\Jm x_{12},\ldots,\Jm x_{m,m-1},\Km x_{12},\ldots,\Km x_{m ,m-1})'.
\end{align*}
For notational convenience, relabel {{$x$}} as $(x_1,\ldots,x_n)$.
Similar to the real case, we have
$$
(ds)^2=\Rtr(X^{-1}dXX^{-1}dX)=dx'G(x)dx
$$
where $G(x)$ is an $n \times n$ nonsingular symmetric matrix. Define the differential operator $\Delta_X^{\ast}$ as
$$
\Delta_X^{\ast}=\det G(x)^{-1/2}\sum_{j=1}^n
\frac{\partial}{\partial x_j}{{\bigg[}}\det G(x)^{1/2}
\sum_{i=1}^ng(x)^{ij}\frac{\partial}{\partial x_i}{{\bigg]}},
$$
where $G(x)^{-1}=(g(x)^{ij})$. Let $\dfrac{\partial}{\partial
x}={{\bigg(}}\dfrac{\partial}{\partial x_1}, \cdots,
\dfrac{\partial}{\partial x_n}{{\bigg)}}'$, then we can write
$\Delta_X^{\ast}$ as
\begin{equation}\label{eqaaa}
\Delta_X^{\ast}=\det G(x)^{-1/2}{{\bigg(}}\frac{\partial}{\partial
x}{{\bigg)}}^{'}{{\bigg[}}\det
G(x)^{1/2}G(x)^{-1}\frac{\partial}{\partial x}{{\bigg]}}
\end{equation}
{{which}} is invariant under the transformation $X\rightarrow LXL^H$
($L\in\Q^{m \times m}$ is invertible), i.e.,
$\Delta_X^{\ast}=\Delta_{LXL^H}^{\ast}$.

The proofs of {{the}} above assertions are just {{the}} same as in
\cite[{p}240]{robb} {{and}} we do not show them here. {{Consider the
positive definite quaternion}} matrix $X=HYH^H$, $H \in {_qO(m)}$,
$Y=\diag(y_1,\ldots,y_m)$. In terms of $H$ and $Y$, the invariant
differential form $(ds)^2$ given by (\ref{eqzzz}) can be written as

\begin{align*}
(ds)^2&=\Rtr(X^{-1}dXX^{-1}dX)\\
&=\Rtr(Y^{-1}dYY^{-1}dY)-2\Rtr(d\Theta Y^{-1}d\Theta Y^{-1})+2\Rtr(d\Theta d\Theta)\\
&=\sum\limits_{i=1}^m\frac{(dy_i)^2}{y_i^2}-2\sum\limits_{i=1}^m
((\Im d\theta_{ii})^2+(\Jm d\theta_{ii})^2+(\Km d\theta_{ii})^2)\\
&+2\sum\limits_{i<j}^m\frac{y_i^2+y_j^2}{y_iy_j}((\Re d\theta_{ij})^2+(\Im d\theta_{ij})^2+(\Jm d\theta_{ij})^2
+(\Km d\theta_{ij})^2)\\
&+2\sum\limits_{i=1}^m ((\Im d\theta_{ii})^2+(\Jm d\theta_{ii})^2+(\Km d\theta_{ii})^2)\\
&-4\sum\limits_{i<j}^m((\Re d\theta_{ij})^2+(\Im d\theta_{ij})^2+(\Jm d\theta_{ij})^2+(\Km d\theta_{ij})^2)\\
&=\sum\limits_{i=1}^m\frac{(dy_i)^2}{y_i^2}+2\sum\limits_{i<j}^m\frac{{(y_i-y_j)}^2}{y_iy_j}
((\Re d\theta_{ij})^2+(\Im d\theta_{ij})^2+(\Jm d\theta_{ij})^2+(\Km d\theta_{ij})^2)\\
&=((dy)^{'}\ (\Re d\theta)^{'}\ (\Im d\theta)^{'}\ (\Jm d\theta)^{'}\ (\Km d\theta')G(y)
\begin{pmatrix}dy\\ \Re d\theta\\ \Im d\theta\\ \Jm d\theta\\ \Km d\theta \end{pmatrix}
\end{align*}
where $d\Theta=(d\theta_{ij})=H^HdH=-dH^HH$,
$dy=(dy_1,dy_2,{\ldots},dy_m)^{'}$, and
\begin{align*}
\Re d\theta=&(\Re d\theta_{12},\Re d\theta_{13},{\ldots},\Re
d\theta_{m-1,m})^{'},\quad
\Im d\theta=(\Im d\theta_{12},\Im d\theta_{13},{\ldots},\Im d\theta_{m-1,m})^{'},\\
\Jm d\theta=&(\Jm d\theta_{12},\Jm d\theta_{13},{\ldots},\Jm
d\theta_{m-1,m})^{'},\ \Km d\theta=(\Km d\theta_{12},\Km
d\theta_{13},{\ldots},\Km d\theta_{m-1,m})^{'}.
\end{align*}
Therefore $G(y)$ has the form
$$
G(y)=\left( \begin{array}{cc}B &0\\0&\begin{array}{ccccc}A_{12}&&&&\\&\ddots&&&\\&&A_{ij}(i<j)&&
\\&&&\ddots&\\&&&&A_{m-1, m}\end{array}\end{array}\right),
$$
where
$$
B=\left(\begin{array}{ccc}y_1^{-2}&&\\& \ddots & \\ &&y_{m}^{-2}\end{array}\right),\
A_{ij}= \left( \begin{array}{cccc}\dfrac{2(y_i-y_j)^2}{y_iy_j}&&&\\&\dfrac{2(y_i-y_j)^2}{y_iy_j}&&\\&&
\dfrac{2(y_i-y_j)^2}{y_iy_j}&\\ &&&\dfrac{2(y_i-y_j)^2}{y_iy_j}\end{array}\right).
$$
In terms of (\ref{eqaaa}) and $\dfrac{\partial}{\partial y}$, $\dfrac{\partial}{\partial R\theta}$,
$\dfrac{\partial}{\partial I\theta}$, $\dfrac{\partial}{\partial J\theta}$,
$\dfrac{\partial}{\partial K\theta}$, the operator $\Delta_X^\ast$ can be expressed as
$$\Delta_X^\ast=\Delta_{HYH^H}^\ast=|G(y)|^{-1/2}\begin{pmatrix}\dfrac{\partial}{\partial y}\\
\dfrac{\partial}{\partial R\theta}\\ \dfrac{\partial}{\partial I\theta}\\ \dfrac{\partial}{\partial
J\theta}\\ \dfrac{\partial}{\partial K\theta}\end{pmatrix}^{'}
\left[|G(y)|^{1/2}G(y)^{-1}\begin{pmatrix}\dfrac{\partial}{\partial y}\\
\dfrac{\partial}{\partial R\theta}\\ \dfrac{\partial}{\partial I\theta}\\ \dfrac{\partial}{\partial J\theta}\\
\dfrac{\partial}{\partial K\theta}\end{pmatrix}\right],
$$
($\dfrac{\partial}{\partial R\theta}$, $\dfrac{\partial}{\partial I\theta}$, $\dfrac{\partial}{\partial J\theta}$,
$\dfrac{\partial}{\partial K\theta}$ are the derivation of $\Re\theta,\Im\theta,\Jm\theta,\Km\theta$ respectively),
that is,
\begin{align*}
\Delta_X^\ast =&\Delta_{HYH^H}^\ast=\sum\limits_{i=1}^my_i^2\frac{\partial^2}{\partial
y_i^2}+4\sum\limits_{i=1}^m\sum\limits_{j=1,j\neq i}^m\dfrac{y_i^2}{y_i-y_j}\dfrac{\partial}{\partial y_i}\\
&+(3-2m)\sum\limits_{i=1}^my_i\dfrac{\partial}{\partial
y_i}+\dfrac{1}{2}\sum\limits_{i<j}^m\frac{y_iy_j}
{(y_i-y_j)^2}{{\bigg(}}\dfrac{\partial^2}{\partial
R\theta_{ij}^2}+\dfrac{\partial^2}{\partial I\theta_{ij}^2}+
\dfrac{\partial^2}{\partial J\theta_{ij}^2}+\dfrac{\partial^2}{\partial K\theta_{ij}^2}{\bigg)}\\
=&\Delta_Y+(3-2m)E_Y+\frac{1}{\,2\,}\sum\limits_{i<j}^m\frac{y_iy_j}{(y_i-y_j)^2}{{\bigg(}}\frac{\partial^2}{\partial
R\theta_{ij}^2}+\dfrac{\partial^2}{\partial
I\theta_{ij}^2}+\dfrac{\partial^2}{\partial
J\theta_{ij}^2}+\dfrac{\partial^2}{\partial K\theta_{ij}^2}{\bigg)}
\end{align*}
where $\Delta_Y$ is given in Definition \ref{def2.1},
$E_Y=\sum\limits_{i=1}^{m}y_i\dfrac{\partial}{\partial y_i}$,
$\dfrac{\partial^2}{\partial R\theta^2},\dfrac{\partial^2}{\partial
I\theta^2},\dfrac{\partial^2}{\partial
J\theta^2},\dfrac{\partial^2}{\partial K\theta^2}$ is the second
derivation of $\Re\theta,~\Im\theta,~\Jm\theta,~\Km\theta${{,}}
respectively. It follows from $E_YC_\kappa(Y)=kC_\kappa(Y)$ and
{{the}} above equation that
\begin{align*}
\Delta_X^\ast& C_\kappa(X)=\Delta_{HYH^H}^\ast C_\kappa(Y)\\
&={{\bigg[}}\Delta_Y+(3-2m)E_Y+\frac{1}{\,2\,}\sum\limits_{i<j}^m\frac{y_iy_j}{(y_i-y_j)^2}{{\bigg(}}\frac{\partial^2}{\partial
R\theta_{ij}^2}+\frac{\partial^2}{\partial
I\theta_{ij}^2}+\frac{\partial^2}{\partial
J\theta_{ij}^2}+\frac{\partial^2}{\partial
K\theta_{ij}^2}{{\bigg)}}{{\bigg]}}C_\kappa(Y)\\
&=[\rho_\kappa +k(4m-1)+(3-2m)k]C_\kappa(Y)\\
&= [\rho_\kappa+2k(m+1)]C_\kappa(X).
\end{align*}

In fact, we could have defined the zonal polynomial $C_{\kappa}(X)$
for $X>0$ in terms of the operator $\Delta_X^{*}$ rather than
$\Delta_Y$. Here the definition would be that
$C_\kappa(X)~(=C_\kappa(Y))$ is a symmetric homogeneous polynomial
of degree $k$ in the latent roots $y_1,\cdots,y_m$ of $X$ satisfying
conditions (i) and (iii) of definition \ref{def2.1} and such that
$C_\kappa(X)$ is an eigenfunction of the differential operator
$\Delta_X^\ast$. The eigenvalue of $\Delta_X^\ast$ corresponding to
$C_\kappa(X)$ is, from {{the}} above equation, equal to
$[\rho_\kappa+2k(m+1)]$. This defines the zonal polynomials for the
{{positive definite quaternion}} matrix $X$, and since they are
polynomials in the latent roots of $X$ their definition can be
extended to an arbitrary Hermitian quaternion matrix and then to
{{a}} non-Hermitian quaternion matrix by using
$C_\kappa(XY)=C_\kappa(X^{1/2}YX^{1/2})$ ($X$ is a positive definite
matrix and $Y$ is a Hermitian matrix).

\begin{theorem}\label{th 2.2}
Let $X_1, X_2\in{_qS(m)}$ with $X_1$ positive definite. Then
$$
\int_{{_qO(m)}}C_\kappa(X_1HX_2H^H)(dH)=\frac{C_\kappa(X_1)C_\kappa(X_2)}{C_\kappa(I_m)},
$$
where $(dH)$ is the normalized invariant measure on $_qO(m)$.
\end{theorem}
\begin{proof} Let
$$
f_\kappa(X_2)=\int_{{_qO(m)}}C_\kappa(X_1HX_2H^H)(dH).
$$
It is easy to verify $f_\kappa(X_2)=f_\kappa(UX_2U^H)$, $U \in {_qO(m)}$ so that $f_\kappa(X_2)$ is a symmetric
function of $X_2$; in fact, a symmetric homogeneous polynomial of degree $k$. Set $L=X_1^{1/2}H$ and suppose $X_2>0$.
Then by use of the invariance of $\Delta^{\ast}_{X_2}$, we have
\begin{align*}
\Delta_{X_2}^\ast f_\kappa(X_2)&=\int_{{_qO(m)}}\Delta_{X_2}^\ast C_\kappa(X_1HX_2H^H)(dH)\\
&=\int_{{_qO(m)}}\Delta_{X_2}^\ast C_\kappa(X_1^{1/2}HX_2H^HX_1^{1/2})(dH)\\
&=\int_{{_qO(m)}}\Delta_{X_2}^\ast C_\kappa(LX_2L^H)(dH) =\int_{{_qO(m)}}\Delta_{LX_2L^H}^\ast C_\kappa(LX_2L^H)(dH)\\
&=[\rho_\kappa+2k(m+1)]f_\kappa(X_2)
\end{align*}
Then $f_\kappa(X_2)$ must be a multiple of the zonal polynomial
$C_\kappa(X_2)$, i.e., $f_\kappa(X_2)=\lambda_\kappa C_\kappa(X_2)$.
Put $X_2=I_m$, then
$\lambda_\kappa=\dfrac{C_\kappa(X_1)}{C_\kappa(I_m)}$. Finally, {{we
get the result by analytic continuation}}.
\end{proof}

Theorem \ref{th 2.2} plays a vital role in the next evaluation of
many integrals involving zonal polynomials.

Let $\displaystyle\Q\Gamma_m(a)=\int_{A>0}\mathrm{etr}\,(-A)|A|^{a-2m+1}(dA)$ be the quaternion $\Gamma$--function
given in \cite{kenn} and then $\mathbb{Q}\Gamma_n(\alpha)=\pi^{n(n-1)}\dprod_{j=1}^{n}\Gamma[\alpha-2(j-1)],~~
\Re\alpha >2(n-1)$. Set
$$
\mathbb{Q}\Gamma_n(\alpha,\kappa)=\pi^{n(n-1)}\dprod_{j=1}^{n}\Gamma[\alpha+k_j-2(j-1)],~~\Re\alpha
>2(n-1)-k_n,
$$
where $\kappa=(k_1,{\ldots},k_n)$ is a partition of the integer $k$:
$k=k_1+k_2+\cdots+k_n$, $k_1\geqslant k_2 \geqslant \cdots \geqslant
k_n\geqslant0$. Then we have
$(\alpha)_\kappa\triangleq\dprod_{j=1}^{n}(\alpha-2(j-1))_{k_j}=\dfrac{\Q\Gamma_n(\alpha,\kappa)}{\Q\Gamma_n(\alpha)}$,
{{where}} $(\alpha)_j=\alpha(\alpha+1)\cdots (\alpha+j-1)$.

\begin{lemma}\label{lemma2.1}
Let $A=(a_{ij})_{m\times m}\in{_qS(m)}$ with eigenvalues
$\lambda_1,{\ldots}, \lambda_m$ (are all real). Put
$r_1=\sum\limits^m_{i=1}\lambda_i$, $r_2=\sum\limits_{i<j}^m
\lambda_i\lambda_j,~\cdots,~r_m=\lambda_1\cdots\lambda_m$ and
$\tr_k(A)=\sum\limits_{1\leq i_1<i_2<\cdots<i_k\leq m}\det
A_{i_1,i_2\cdots i_k}$, where $A_{i_1,i_2,\cdots,i_k}$ {{denotes}}
the $k\times k$ matrix formed from $A$ by deleting all but the
$i_1,{\ldots},i_k$th rows and columns. Then $r_j=\tr_j(A)$.
\end{lemma}
\begin{proof}
We have $P(\lambda)=|A-\lambda
I_m|=\sum\limits_{k=0}^{m}(-\lambda)^kr_{m-k}(\lambda_1,{\ldots},\lambda_m)$.
{{We}} also can get $ |A-\lambda
I_m|=\sum\limits_{k=0}^{m}(-\lambda)^k\tr_{m-k}(A)$ by the
definition of the determinant of a quaternion matrix given in
\cite{zh}. The assertion follows.
\end{proof}

\begin{lemma}\label{lemma2.2}
 Let $Y=\diag(y_1,y_2,{\ldots},y_m)$ be a real diagonal matrix
and $X=(x_{ij})_{m\times m}$ be a $m \times m$ positive definite
quaternion matrix. Then
\begin{equation}
C_{\kappa}(XY)=d_{\kappa}y_1^{k_1}\cdots y_m^{k_m}x_{11}^{k_1-k_2}
\Big|{\begin{pmatrix}x_{11}&x_{12}\\ x_{21}&x_{22} \end{pmatrix}}\Big|^{k_2-k_3}\cdots
|X|^{k_m}\\+\textit{terms of lower weight},
\end{equation}
where $\kappa=(k_1,\cdots,k_m),~d_{\kappa}$ is the coefficient of the term of highest weight in $C_{\kappa}(\cdot)$.

If $Z=\diag(z_1,z_2,{\ldots},z_m)$ is a real diagonal matrix and
$Y=(y_{ij})_{m\times m}$ is a $m \times m$ positive definite
quaternion matrix, then
\begin{align*}
C_{\kappa}(Y^{-1}Z)=&d_{\kappa}z_1^{k_m}\cdots
z_m^{k_1}y_{11}^{k_{m-1}-k_m}\Big|{\begin{pmatrix}y_{11}&y_{12}\\ y_{21}&y_{22} \end{pmatrix}}\Big|^{k_{m-2}-k_{m-1}}
\cdots |Y|^{-k_1}\\
&+\textit{terms of lower weight},
\end{align*}
where $\kappa=(k_1,{\ldots},k_m),~d_\kappa$ is the coefficient of
the term of highest weight in $C_{\kappa}(\cdot)$.
\end{lemma}

\begin{proof}
Let $A\in{_qS(m)}$ and $a_1,{\ldots}, a_m$ be its real eigenvalues.
Then
\begin{align*}
C_\kappa(A)=& d_\kappa a_1^{k_1}\cdots a_m^{k_m}+ \textit{terms of lower weight} \\
=& d_{\kappa}a_1^{k_1-k_2}(a_1a_2)^{k_2-k_3}\cdots (a_1a_2\cdots a_m)^{k_m}+\textit{terms of lower weight} \\
=&d_{\kappa} (\sum_{i=1}^m a_i)^{k_1-k_2}(\sum_{i<j}^{m}a_ia_j)^{k_2-k_3}\cdots (a_1a_2\cdots
a_m)^{k_m}+{\textit{symmetric terms}}\\
=& d_{\kappa}r_1^{k_1-k_2}r_2^{k_2-k_3}\cdots
r_m^{k_m}+{\textit{symmetric terms}}.
 \end{align*}
On the other hand, by Lemma \ref{lemma2.1},
\begin{align*}
C_{\kappa}(A)=&d_{\kappa}\tr_1(A)^{k_1-k_2}\tr_2(A)^{k_2-k_3}\cdots \tr_m(A)+{\textit{symmetric terms}} \\
=&d_{\kappa}a_{11}^{k_1-k_2}{\Big|\begin{pmatrix}a_{11}&a_{12}\\ a_{21}&a_{22} \end{pmatrix}}\Big|^{k_2-k_3}\cdots
\end{align*}
Set $A=XY,~a_{ij}=x_{ij}y_j$. We have
$$
C_{\kappa}(XY)=d_{\kappa}y_1^{k_1}\cdots y_m^{k_m}x_{11}^{k_1-k_2}\Big|
{\begin{pmatrix}x_{11}&x_{12}\\x_{21}&x_{22} \end{pmatrix}}\Big|^{k_2-k_3}\cdots
|X|^{k_m}\\+\textit{terms of lower weight}.
$$

Similarly, we can get the second assertion.
\end{proof}

Let $A=A_1+A_2i+A_3j+A_4k\in\Q^{m\times n}$ and put $\re(A)=A_1+A_2i+A_3j$. Let $\Phi_m=\{T\in{_qS(m)}\vert\,
{\re(T)>0\}}$. $\Phi_m$ is called the generalized right half plane.
\begin{theorem}\label{th2.3}
Let $Z\in\Phi_m$ and $Y\in{_qS(m)}$. Then
$$
\int_{X>0}\etr(-XZ)|X|^{a-2m+1}C_{\kappa}(XY)(dX)=(a)_{\kappa}\mathbb{Q}\Gamma_m(a)|Z|^{-a}C_{\kappa}(YZ^{-1}),
$$
for $\Re(a)>2(m-1)$ and
$$
\int_{X>0}\etr(-XZ)|X|^{a-2m+1}C_{\kappa}(X^{-1}Y)(dX)=\frac{(-1)^{k}\Q\Gamma_m(a)}{(-a+2m-1)_{\kappa}}|Z|^{-a}
C_{\kappa}(YZ)
$$
for $\Re(a)>2(m-1)+k_1$, where we set $C_{\kappa}=1$ {{and}}
$(a)_\kappa=1$ when $\kappa=(0)$.
\end{theorem}
\begin{proof}
For $Z=I_m$, we should prove the following equation
$$
\int_{X>0}\etr(-X)|X|^{a-2m+1}C_{\kappa}(XY)(dX)=(a)_{\kappa}\mathbb{Q}\Gamma_m(a)C_{\kappa}(Y).
$$
Let $\displaystyle f(Y)=\int_{X>0}\etr(-X)|X|^{a-2m+1}C_{\kappa}(XY)(dX)$ and put $S=H^HXH$, $H\in {_qO(m)}$. Then
$(dS)=(dX)$ and
\begin{align*}
f(HYH^H)=&\int_{X>0}\etr(-X)|X|^{a-2m+1}C_{\kappa}(XHYH^H)(dX)\\
=&\int_{S>0}\etr(-S)|S|^{a-2m+1}C_{\kappa}(SY)(dS)=f(Y)
\end{align*}
and hence
\begin{align*}
f(Y)=&\int_{{_qO(m)}}f(Y)(dH)=\int_{{_qO(m)}}f(HYH^H)(dH)\\
=&\int_{X>0}\etr(-X)|X|^{a-2m+1}\int_{{_qO(m)}}C_{\kappa}(XHYH^H)(dH)(dX)\\
=&\int_{X>0}\etr(-X)|X|^{a-2m+1}\frac{C_{\kappa}(X)C_{\kappa}(Y)}{C_{\kappa}(I_m)}(dX)\\
=&\frac{C_{\kappa}(Y)}{C_{\kappa}(I_m)}f(I_m).
\end{align*}
Since $f(Y)$ is a symmetric homogeneous polynomial in the latent of
$Y$, it can be assumed without loss of generality that $Y$ is diagonal, $Y=\diag(y_1,\ldots,y_m)$, using (i) of
Definition \ref{def2.1}, $f(Y)=\dfrac{f(I_m)}{C_{\kappa}(I_m)}d_{\kappa}y_1^{k_1}\cdots y_m^{k_m}+\cdots$, since
\begin{align*}
f(Y)=&\int_{X>0}\etr(-X)|X|^{a-2m+1}d_{\kappa}y_1^{k_1}\cdots y_m^{k_m}\times \\
&x_{11}^{k_1-k_2}\Big|\begin{pmatrix}x_{11}&x_{12}\\ x_{21}&x_{22}\end{pmatrix}\Big|^{k_2-k_3}\cdots |X|^{k_m}(dX).
\end{align*}
Put $X=T^HT$, {{where}} $T$ is a upper triangular with positive
diagonal elements. Then
$$
\tr X=\sum_{i\leq j}^{m}t_{ij}^Ht_{ij},\ x_{11}=t_{11}^2,\
\Big|\begin{pmatrix}x_{11}&x_{12}\\ x_{21}&x_{22}\end{pmatrix}\Big|=t_{11}^2t_{22}^2,\cdots,\
|X|=\dprod_{i=1}^{m}t_{ii}^2
$$
(Lemma \ref{Lax} (2)). By {{Lemma}} \ref{lemma1.3},
\begin{align*}f(Y)=&\int_{X>0}\exp {{\bigg(}}-\sum\limits_{i\leq j}^{m}t_{ij}^Ht_{ij}{{\bigg )}}\dprod_{i=1}^{m}t_{ii}^{2a-4m+2}
d_{\kappa}y_{1}^{k_1}\cdots {y_m^{k_m}}\\
&\ {{\times}}\prod_{i=1}^{m}{t_{ii}^{2k_i}}2^m\dprod_{i=1}^mt_{ii}^{4m-4i+1}\dwedge_{i\leq j}^{m}dt_{ij}+\cdots\\
=&d_{\kappa}y_{1}^{k_1}\cdots y_m^{k_m}\pi^{m(m-1)}\prod_{i=1}^{m}\Gamma(a+k_i-2(i-1))+\cdots \\
=&d_{\kappa}y_{1}^{k_1}\cdots y_m^{k_m}(a)_{\kappa}\Q\Gamma_m(a)+\cdots.
\end{align*}
By comparing the coefficients of {{the}} two {{expressions}} of
$f(Y)$, we have
$\dfrac{f(I_m)}{C_{\kappa}(I_m)}=(a)_{\kappa}\mathbb{Q}\Gamma_m(a)$.

When $Z>0$, let $V=Z^{{1/2}}XZ^{{1/2}}$. Then
$(dV)=|Z|_q^{2m-1}(dX)$ and
\begin{align*}
\int_{X>0}&\etr(-XZ)(\det X)^{a-2m+1}C_{\kappa}(XY)(dX)\\
=&|Z|_q^{-a}\int_{X>0}\etr(-Z^{-{1/2}}VZ^{{1/2}})|V|^{a-2m+1}
C_{\kappa}(VZ^{-{1/2}}YZ^{-{1/2}})(dV)\\
=&|Z|_q^{-a}\int_{X>0}\etr(-V)|V|^{a-2m+1}C_{\kappa}(VZ^{-{1/2}}YZ^{-{1/2}})(dV)\\
=&|Z|_q^{-a}(a)_{\kappa}\mathbb{Q}\Gamma_m(a)C_{\kappa}(YZ^{-1}).
\end{align*}

Finally, by analytic continuation, we get {{the}} result on $\Phi_m$
since the left--side of the integrations in the {{theorem}} is
absolutely convergent in $\Phi_m$.
\end{proof}
\begin{definition}
If $f(X)$ is a function of the positive definite $m \times m$
quaternion matrix $X$, the {{Laplace}} transform of $f(X)$ is
defined to be
$$
g(Z)=\mathcal{L}(f(X))=\int_{X>0}\etr(-XZ)f(X)(dX)
$$
which is absolutely convergent for $Z\in\Phi_m$. Note that $\mathcal{L}(\cdot)$ is one to one for $Z-P\in\Phi_m$
where $P$ is a complex positive definite matrix \rm{(cf.\ \cite{kenn})}.
\end{definition}

\begin{theorem}\label{th2.4}
Let $Y\in {_qS(m)}$. Then
$$
\int_{0<X<I_m}|X|^{a-2m+1}|I-X|^{b-2m+1}C_{\kappa}(XY)(dX)=
\frac{\mathbb{Q}\Gamma_m(a,\kappa)\mathbb{Q}\Gamma_m(b)}{\mathbb{Q}\Gamma_{m}(a+b,\kappa)}C_{\kappa}(Y)
$$
for $\Re(a)>2(m-1)$, $\Re(b)>2(m-1)$ and
$$
\int_{0<X<I_m}|X|^{a-2m+1}|I-X|^{b-2m+1}C_{\kappa}(X^{-1}Y)(dX)=
\frac{\mathbb{Q}\Gamma_m(a,-\kappa)\mathbb{Q}\Gamma_m(b)}{\mathbb{Q}\Gamma_{m}(a+b,-\kappa)}C_{\kappa}(Y)
$$
for $\Re(a)>2(m-1)+k_1$, $\Re(b)>2(m-1)$.
\end{theorem}
\begin{proof}
Let
$f(Y)=\displaystyle\int_{0<X<I_m}|X|^{a-2m+1}|I-X|^{b-2m+1}C_{\kappa}(XY)(dX)$.
It is easy to check that
$f(Y)=f(HYH^H),~H\in{_qO(m)}$ and $f(Y)C_{\kappa}(I_m)=f(I_m)C_{\kappa}(Y)$ by Theorem \ref{th 2.2}. Take $Z=I_m$ and $Y=I_m$ in
Theorem \ref{th2.3}. Then
\begin{align*}
\int_{W>0}\etr(-W)|W|^{a+b-2m+1}f(W)(dW)
=&\int_{W>0}\etr(-W)|W|^{a+b-2m+1}\frac{f(I_m)C_{\kappa}(W)}{C_{\kappa(I_m)}}(dW)\\
=&\frac{f(I_m)}{C_{\kappa(I_m)}}\mathbb{Q}\Gamma_{m}(a+b,\kappa)C_{\kappa}(I_m)\\
=&f(I_m)\mathbb{Q}\Gamma_{m}(a+b,\kappa).
\end{align*}
Set $X=W^{-{1/2}}UW^{-{1/2}}$. Then
\begin{align*}
&\int_{W>0}\etr(-W)|W|^{a+b-2m+1}f(W)(dW)\\
&=\int_{W>0}\etr(-W)|W|^{a+b-2m+1}\int_{0<X<I_m}|X|^{a-2m+1}|I-X|^{b-2m+1}C_{\kappa}(XW)(dX)(dW)\\
&=\int_{W>0}\etr(-W)|W|^{a+b-2m+1}\int_{0<U<W}|U|^{a-2m+1}|W|^{-a-b+4m-2}|W-U|^{b-2m+1}\\
&\quad\quad\times C_{\kappa}(W^{{1/2}}UW^{-{1/2}})|W|^{1-2m}(dU)(dW)\\
&=\int_{U>0}\etr(-V-U)\int_{V>0}|U|^{a-2m+1}|V|^{b-2m+1}C_{\kappa}(U)(dV)(dU)\ (\mbox{for}\ V=W-U)\\
&=\int_{U>0}\etr(-U)|U|^{a-2m+1}C_{\kappa}(U)(dU)\int_{V>0}\etr(-V)|V|^{b-2m+1}(dV)\\
&=\Q\Gamma_m(a,\kappa)\Q\Gamma_m(b)C_{\kappa}(I_m).
\end{align*}
So $f(I_m)=\dfrac{\Q\Gamma_m(a,\kappa)\Q\Gamma_m(b)}{\Q\Gamma_m(a+b,\kappa)}C_{\kappa}(I_m)$ and hence
$f(Y)=\dfrac{\Q\Gamma_m(a,\kappa)\Q\Gamma_m(b)}{\Q\Gamma_m(a+b,\kappa)}C_{\kappa}(Y)$.
\end{proof}
\begin{corollary}\label{coro2.1} If $Y\in{_qS(m)}$, then
$$\int_{0<X<I_m}|X|^{a-2m+1}C_{\kappa}(XY)(dX)=\frac{(a)_{\kappa}}{(a+2m-1)_\kappa}
\frac{\mathbb{Q}\Gamma_m(a)\mathbb{Q}\Gamma_m(2m-1)}{\mathbb{Q}\Gamma_{m}(a+2m-1)}C_{\kappa}(Y)$$
{{where}} $\Re(a)>2(m-1)$, { and} $\kappa=(k_1,k_2,{\ldots},k_m)$.
\end{corollary}

\section{HYPERGEOMETRIC FUNCTION FOR QUATERNION MATRIX}

\begin{definition}
The hypergeometric functions of {{a}} Hermitian quaternion matrix
argument are given by
\begin{equation}\label{equyy}
{_pF_q(a_1,\cdots,a_p;b_1,\cdots,b_q;X)}=\sum_{k=0}^{\infty}\sum_{\kappa}\frac{(a_1)_{\kappa}\cdots(a_p)_{\kappa}}
{(b_1)_{\kappa}\cdots(b_q)_{\kappa}}\frac{C_{\kappa}(X)}{k!}
\end{equation}
where $\sum\limits_{\kappa}$ denotes summation over all partitions
$\kappa=(k_1,{\ldots},k_m)$, $k_1\geqslant \cdots \geqslant k_m
\geqslant 0$ of $k$ and $X\in{_qS(m)}$.
\end{definition}

\begin{remark}
We have the special case ${_0F_0(A)}=\etr A$ for $A\in{_qS(m)}$. From \cite{kenn}, we have
\begin{enumerate}
\item[\rm{(1)}] If $p<q$, then the hypergeometric series (\ref{equyy}) converges absolutely for all $X$;
\item[\rm{(2)}] If $p=q+1$, then the series (\ref{equyy}) converges absolutely for $\|X\|<1$ and diverges for $\|X\|>1$;
\item[\rm{(3)}] If $p>q$, then the series (\ref{equyy}) diverges unless it terminates.
\end{enumerate}
\end{remark}

\begin{definition}
The hypergeometric functions of Hermitian quaternion matrices $X$,
$Y$ are given by
\begin{equation}{_pF_q}^{m}(a_1,\cdots,a_p;b_1,\cdots,b_q;X,Y)=\sum_{k=0}^{\infty}\sum_{\kappa}\frac{(a_1)_{\kappa}\cdots(a_p)_{\kappa}}
{(b_1)_{\kappa}\cdots(b_q)_{\kappa}}\frac{C_{\kappa}(X)C_{\kappa}(Y)}{C_{\kappa}(I_m)k!}\end{equation}
\end{definition}

By Theorem \ref{th 2.2}, we have
\begin{theorem}\label{th3.1}
If $X,\ Y\in{_qS(m)}$ with $X>0$, then
$$
\int_{_qO(m)}{_pF_q}(a_1,\cdots,a_p;b_1,\cdots,b_q;XHYH^H)(dH)={_pF_q}^{m}(a_1,\cdots,a_p;b_1,\cdots,b_q;X,Y).
$$
\end{theorem}
By Theorem \ref{th2.3}, we also have
\begin{theorem}\label{th3.2}
Let $Z\in{_qS(m)}$ and suppose $p\leqslant q$, $\Re(a)>2(m-1)$. Then
\begin{align*}
\int_{X>0}\etr(-XZ)&(\det X)^{a-2m+1}{_pF_q}(a_1,\cdots,a_p;b_1,\cdots,b_q;X)(dX)\\
&=\Q\Gamma_m(a)(\det Z)^{-a}{_{p+1}F_q}(a_1,\cdots,a_p,a;b_1,\cdots,b_q;Z^{-1})
\end{align*}
and
\begin{align*}
\int_{X>0}\etr(-XZ)&(\det X)^{a-2m+1}{_pF_q}^{m}(a_1,\cdots,a_p;b_1,\cdots,b_q;X,Y)(dX)\\
&=\Q\Gamma_m(a)(\det Z)^{-a}{_{p+1}F_q}^{m}(a_1,\cdots,a_p,a;b_1,\cdots,b_q;Z^{-1},Y)
\end{align*}
for all $Z\in \Phi_m$ when $p<q$ and  for
$\|[\re(Z)]^{-1}\|<1$ when $p=q$.
\end{theorem}

\begin{corollary}\label{Coroz}
Let $Z\in{_qS(m)}$ with $\|Z\|<1$ and $\Re(a)>2(m-1)$. Then ${_1F_0}(a;Z)=|I_m-Z|^{-a}$.
\end{corollary}
\begin{proof}
Assume that $0<Z<I_m$. By {{Theorem \ref{th3.2}}},
$$\int_{X>0}\etr(-XZ^{-1})|X|^{a-2m+1}\etr(X)(dX)=\mathbb{Q}\Gamma_{m}(a)|Z|^a{_1F_0}(a,Z).$$
Let $X=Z^{{1/2}}UZ^{{1/2}}$, then $(dX)=|Z|^{2m-1}(dU)$ by Lemma
\ref{Lbx} (2) and hence
\begin{align*}
\int_{X>0}\etr(-XZ^{-1})|X|^{a-2m+1}&\etr(X)(dX)\\
=&\int_{X>0}\etr(X(I-Z^{-1}))|X|^{a-2m+1}(dX)\\
=&|Z|^{a}\int_{U>0}\etr(-U(I-Z))|U|^{a-2m+1}(dU).
\end{align*}

Put $P=(I-Z)^{{1/2}}U(I-Z)^{{1/2}}$. Then
\begin{align*}
\int_{U>0}&\etr(-U(I-Z))|U|^{a-2m+1}(dU)\\
=&\int_{P>0}|I-Z|^{-a+2m-1}|P|^{a-2m+1}\etr(-P)|I-Z|^{-2m+1}(dP)\\
=&|I-Z|^{-a}\mathbb{Q}\Gamma_{m}(a).
\end{align*}
Finally, we have ${_1F_0}(a;Z)=|I_m-Z|^{-a}$ for $Z\in{_qS(m)}$ with $\|Z\|<1$, by analytic continuity.
\end{proof}

\begin{theorem}
Let $X\in\Q^{m\times n}$ $(m\leq n)$ and $H=(H_1|H_2)\in {_qO(n)}$, $H_1\in{_qV_{m,n}}$. Then
$\displaystyle
{_0F_1}(2n,4XX^H)=\int_{_qO(n)}\exp(4\Rtr(XH_1))(dH).
$
\end{theorem}
\begin{proof}
We {{ use }} the same method as in the proof of \cite[Theorem
7.4.1]{robb}. Assume that $\rank X=m$. Applying the Laplace
transform to $\displaystyle
|X|_q^{2n-2m+1}\int_{_qO(n)}\exp(4\Rtr(XH_1))(dH)$ and
$|X|_q^{2n-2m+1}{_0F_1}(2n,4XX^H)${,} respectively, we have
\begin{align*}
g_l(Z)=&\int_{XX^H>0}\etr(-XX^HZ)|X|_q^{2n-2m+1}\int_{_qO(n)}\exp(4\Rtr(XH_1))(dH)(dXX^H)\\
g_r(Z)=&\int_{XX^H>0}\etr(-XX^HZ)|X|_q^{2n-2m+1}{_0F_1}(2n,4XX^H)(dXX^H).
\end{align*}
Since $(dX)=2^{-m}|X|_q^{2n-2m+1}(dXX^H)(U_1^HdU_1)$, it follows
that
$$
g_l(Z)=\frac{\Q\Gamma_m(2n)}{\pi^{2mn}}\int_{XX^H>0}\int_{_qO(n)}\etr(-XX^HZ)\exp(4\Rtr(XH_1))(dH)(dX).
$$
Let $Z>0$ and put $X=Z^{{-1/2}}Y$. Then $(dX)=|Z|_q^{-n}(dY)$ and
hence
\begin{align*}
g_l(Z)=&\frac{\Q\Gamma_m(2n)}{|Z|_q^n\pi^{2mn}}\int_{YY^H>0}\int_{_qO(n)}
\etr(2(YH_1Z^{{-1/2}}+Z^{{-1/2}}H_1^HY^H)-YY^H)(dH)(dY)\\
=&\frac{\Q\Gamma_m(2n)}{|Z|_q^n\pi^{2mn}}\etr(4Z^{-1})\int_{YY^H>0}\int_{_qO(n)}
\etr(-(Y-2Z^{{-1/2}}H_1^H)(Y-2Z^{{-1/2}} H_1^H)^H)(dH)(dY).
\end{align*}
Note
$\dfrac{1}{\pi^{2mn}}\etr(-(Y-2Z^{{-1/2}}H_1^H)(Y-2Z^{{-1/2}}H_1^H)^H)$
is the density function of $\Q N_{m\times n}(2Z^{{-1/2}}H_1^H,$
$2I_m\otimes I_n)$. Thus
$g_l(Z)={\Q\Gamma_m(2n)}|Z|_q^{-n}\etr(4Z^{-1})$.

On the other hand, by Theorem \ref{th3.2}
\begin{align*}
g_r(Z)=&\Q\Gamma_m(2n)\det(Z)^{-2n}{_1F_1}(2n,2n,4Z^{-1})\\
=&\Q\Gamma_m(2n)|Z|_q^{-n}{_0F_0}(4Z^{-1})\\
=&\Q\Gamma_m(2n)|Z|_q^{-n}\etr(4Z^{-1}).
\end{align*}
Then $g_l(Z)=g_r(Z)$, $\forall\,Z\in\Phi_m$ by analytic continuation.
\end{proof}

\section{THE DISTRIBUTION OF EIGENVALUES}

The joint density function of the eigenvalues of complex central
{{Wishart}} matrix is given in \cite{jacobian} and its
{{distribution}} of the maximum and the minimum eigenvalues is shown
in \cite{trm}. In this section, we generalize some results in
\cite{trm, jacobian} to the quaternion cases.

Let $W=AA^H\sim\Q W_m(n,\Sigma)~~(n\geqslant m)$, $A \sim\Q N(0,I_n\otimes \Sigma)$. The density function of $W$ is
given by (\ref{equaaa}). Let $W=VDV^H$. Then $(dW)=(2\pi^2)^{-m}\dprod^m_{i<j}(\lambda_i-\lambda_j)^4(dD)\bigwedge
(V^HdV)$ by Lemma \ref{lemma1.5}. Then the differential form of the density of $W$ is
$$
\frac{2^{2mn}}{\Q\Gamma_m(2n)|\Sigma|^{2n}}\exp(\Rtr(-2\Sigma^{-1}W))|W|^{2n-2m+1}
(2\pi^2)^{-m}\prod^m_{i<j}(\lambda_i-\lambda_j)^4(dD)\bigwedge (V^HdV).
$$
Integrating {{the}} above equation on $(V^HdV)$, {{by Theorem
\ref{th3.1} we have}}

\begin{align*}
\int&\frac{2^{2mn}}{\Q\Gamma_m(2n)|\Sigma|^{2n}}\exp(\Rtr(-2\Sigma^{-1}W))|W|^{2n-2m+1}(2\pi^2)^{-m}\dprod^m_{i<j}
(\lambda_i-\lambda_j)^4(dD)\dwedge(V^HdV)\\
&=\frac{2^{m}\pi^{2m^2-2m}}{\Q\Gamma_m(2m)|\Sigma|^{2n}}\int\frac{2^{2mn}}{\Q\Gamma_m(2n)}
\exp(\Rtr(-2\Sigma^{-1}W))|W|^{2n-2m+1}\dprod^m_{i<j}(\lambda_i-\lambda_j)^4(dD)\dwedge (dV)\\
&=\frac{2^{2mn}\pi^{2m^2-2m}}{\Q\Gamma_m(2m)\Q\Gamma_m(2n)|\Sigma|^{2n}}{_0F_0}(-2\Sigma^{-1},D)|D|^{2n-2m+1}\dprod^m_{i<j}(\lambda_i-\lambda_j)^4(dD)
\end{align*}
which gives the joint density of the eigenvalues. When $\Sigma=\sigma^2 I_n $, the joint density of the eigenvalues
of $W$ is
\begin{equation}\label{equttt}
\frac{2^{2mn}\pi^{2m^2-2m}}{\Q\Gamma_m(2m)\Q\Gamma_m(2n)|\sigma^2|^{2nm}}|D|^{2n-2m+1}
\dprod^m_{i<j}(\lambda_i-\lambda_j)^4\exp\bigg(-\frac{1}{2\sigma^2}\sum_{i=1}^m\lambda_i\bigg)(dD)
\end{equation}

Let $W\sim\Q W_m(n,\Sigma)(n\geqslant m)$ and $\Delta$ be a $m\times m$ positive definite quaternion matrix.
We will present the distributions of $P(W>\Delta)$ and $P(W<\Delta)$ as follows
\begin{theorem}\label{th4.1}
Let $W$ and $\Delta$ be as above. Then
\begin{align*}
P(W<\Delta)=&\frac{2^{2mn}\Q\Gamma_m(2m-1)}{\Q\Gamma_m(2n+2m-1)}
\frac{|\Delta|^{2n}}{|\Sigma|^{2n}}{_1F_1}(2n,2n+2m-1,-2\Sigma^{-1}\Delta)\\
P(W>\Delta)=&\sum_{k=0}^{m(2n-2m+1)}\widehat{\sum_{\kappa}}\frac{C_{\kappa}(2\Sigma^{-1}\Delta)}{k!}
\etr(-2\Sigma^{-1}\Delta),
\end{align*}
where $\widehat{\sum}$ denotes summation over the partitions
$\kappa=(k_1,{\ldots},k_m)$ of $k$ with $k_1\leqslant 2n-2m+1$.
\end{theorem}

\begin{proof} By means of the density function of $W$ in (\ref{equaaa}), we have
$$
P(W<\Delta)=\frac{2^{2mn}}{\Q\Gamma_m(2n)|\Sigma|^{2n}}\int_{0<W<\Delta}\exp(\Rtr(-2\Sigma^{-1}W))|W|^{2n-2m+1}(dW)$$
Let $W=\Delta^{{1/2}}X\Delta^{{1/2}}$. Then
$(dW)=|\Delta|^{2m-1}dX$. By Corollary \ref{coro2.1}, we get that
\begin{align*}
&P(W<\Delta)=P(X<I)\\
&=\frac{2^{2mn}}{\mathbb{Q}\Gamma_m(2n)|\Sigma|^{2n}}\int_{0<X<I}\exp(\Rtr(-2\Sigma^{-1}
\Delta^{{1/2}}X\Delta^{{1/2}})|\Delta|^{2n-2m+1}
|X|^{2n-2m+1}|\Delta|^{2m-1}(dX)\\
&=\frac{2^{2mn}}{\mathbb{Q}\Gamma_m(2n)}\frac{|\Delta|^{2n}}{|\Sigma|^{2n}}\int_{0<X<I}\etr(-2\Sigma^{-1}
\Delta^{{1/2}}X\Delta^{{1/2}}))|X|^{2n-2m+1}(dX)\\
&=\frac{2^{2mn}}{\mathbb{Q}\Gamma_m(2n)}\frac{|\Delta|^{2n}}{|\Sigma|^{2n}}\int_{0<X<I}\sum_{k=0}^{\infty}
\sum_{|\kappa|=k}\frac{C_{\kappa}(-2\Delta^{{1/2}}\Sigma^{-1}\Delta^{{1/2}}X)}{k!}|X|^{2n-2m+1}(dX)\\
&=\frac{2^{2mn}}{\mathbb{Q}\Gamma_m(2n)}\frac{|\Delta|^{2n}}{|\Sigma|^{2n}}\sum_{k=0}^{\infty}\sum_{|\kappa|=k}
\frac{\mathbb{Q}\Gamma_m(2n)\mathbb{Q}\Gamma_m(2m-1)}{\mathbb{Q}\Gamma(2n+2m-1)}\frac{C_{\kappa}
(-2\Sigma^{-1}\Delta)}{k!}\frac{(2n)_{\kappa}}{(2n+2m-1)_{\kappa}}\\
&=\frac{2^{2mn}\mathbb{Q}\Gamma_m(2m-1)}{\mathbb{Q}\Gamma_m(2n+2m-1)}
\frac{|\Delta|^{2n}}{|\Sigma|^{2n}}{_1F_1}(2n,2n+2m-1,-2\Sigma^{-1}\Delta).
\end{align*}
Note that
$$
P(W>\Delta)=\frac{2^{2mn}}{\Q\Gamma_m(2n)|\Sigma|^{2n}}\int_{W>\Delta}\etr(-2\Sigma^{-1}W)|W|^{2n-2m+1}(dW).
$$
Put $W=\Delta^{{1/2}}(I+X)\Delta^{{1/2}}$. Then
$dW=|\Delta|^{2m-1}(dX)$ and so
\begin{align*}
&P(W>\Delta)\\
&=\frac{2^{2mn}|\Delta|^{2n}}{\Q\Gamma_m(2n)|\Sigma|^{2n}}\int_{X>0}
\etr(-2\Sigma^{-1}\Delta)\etr(-2\Sigma^{-1}\Delta^{{1/2}}X\Delta^{{1/2}})|I+X|^{2n-2m+1}(dX)\\
&=\frac{2^{2mn}|\Delta|^{2n}}{\Q\Gamma_m(2n)|\Sigma|^{2n}}\int_{X>0}\etr(-2\Sigma^{-1}\Delta)
\times\etr(-2\Sigma^{-1}\Delta^{{1/2}}X\Delta^{{1/2}})\\
&\hspace{4cm} \times|I+X^{-1}|^{2n-2m+1}|X|^{2n-2m+1}(dX).
\end{align*}
Since
\begin{align*}
|I+X^{-1}|^{2n-2m+1}=\,&{_1F_0}(-2n+2m-1,-X^{-1})\\
=&\sum_{k=0}^{m(2n-2m+1)}\widehat{\sum_{\kappa}}\frac{[-(2n-2m+1)]_\kappa C_{\kappa}(X^{-1})(-1)^k}{k!}
\end{align*}
by Corollary \ref{Coroz}, it follows from Theorem \ref{th2.3} that
\begin{align*}
\int_{X>0}\etr(-2&\Sigma^{-1}\Delta)\etr(-2\Sigma^{-1}\Delta^{{1/2}}X\Delta^{{1/2}})
|I+X^{-1}|^{2n-2m+1}|X|^{2n-2m+1}(dX)\\
=&\sum_{k=0}^{m(2n-2m+1)}\widehat{\sum}_{\kappa}\frac{(-1)^k[-2n+2m-1]_\kappa}{k!}\\
&\
\times\int_{X>0}\etr(-2\Sigma^{-1}\Delta^{{1/2}}X\Delta^{{1/2}})|X|^{2n-2m+1}C_{\kappa}(X^{-1})
(dX)\\
=&\sum_{k=0}^{m(2n-2m+1)}\widehat{\sum}_{\kappa}\frac{\Q\Gamma_{2m}(2n)}{k!}
|2\Delta^{{1/2}}\Sigma^{-1}\Delta^{{1/2}}|^{-2n}C_\kappa(2\Sigma^{-1}\Delta).
\end{align*}

Therefore, we obtain the result.
\end{proof}
\begin{corollary}
Let $W \sim \Q W_m(n,\Sigma)~~(n\geqslant m)$ and let
$\lambda_{{\mathrm{max}}}$ and $\lambda_{{\mathrm{min}}}$ be the
largest and smallest eigenvalue of $W$ respectively. Then
distribution of $\lambda_{{\mathrm{max}}}$ (resp.
$\lambda_{{\mathrm{min}}}$) is given by
\begin{align}
\label{corx}P(\lambda_{{\mathrm{max}}}<x)=&\frac{\mathbb{Q}\Gamma_m(2m-1)}{\mathbb{Q}\Gamma_m(2n+2m-1)}
\frac{x^{2mn}}{|\Sigma|^{2n}}{_1F_1}(2n,2n+2m-1,-2x\Sigma^{-1})\\
\label{cory}P(\lambda_{{\mathrm{min}}}>x)=&\sum_{k=0}^{m(2n-2m+1)}\widehat{\sum_{\kappa}}
\frac{C_{\kappa}(2x\Sigma^{-1})}{k!}\etr(-2x\Sigma^{-1}).
\end{align}
The density of $\lambda_{{\mathrm{max}}}$ (resp.
$\lambda_{{\mathrm{min}}}$) is obtained by differentiating
(\ref{corx}) (resp. (\ref{cory})) with respect to $x$.
\end{corollary}
\begin{proof}
The inequality $\lambda_{{\mathrm{max}}}<x$ (resp.
$\lambda_{{\mathrm{min}}}>x$) is equivalent to $W<xI_m$ (resp.
$W>xI_m$). The assertions follow by taking $\Delta=xI_m$ in Theorem
\ref{th4.1}.
\end{proof}

\noindent{\bf{Acknowledgement.}}
The authors are grateful to the referee for his (or her) helpful comments and kindly pointing out many typos in
the paper.

\end{document}